\documentclass[12pt]{article}
\usepackage{e-jc_noname}  
\usepackage{amsthm,amsmath,amssymb}
\usepackage{graphicx}
\usepackage[colorlinks=true,citecolor=black,linkcolor=black,urlcolor=blue]{hyperref}

\theoremstyle{plain}
\newtheorem{theorem}{Theorem}
\newtheorem{lemma}[theorem]{Lemma}

\theoremstyle{definition}

\theoremstyle{remark}

\usepackage[T1]{fontenc} 
\usepackage[utf8]{inputenc}
\usepackage{enumerate,enumitem,comment}
\usepackage{mathrsfs}
\usepackage{cmtt}

\usepackage[noline,linesnumbered,ruled,titlenumbered,noend,norelsize]{algorithm2e}
\SetAlgoNoLine
\DontPrintSemicolon
\SetNlSty{footnotesize}{}{}
\IncMargin{0.5em}
\SetNlSkip{1em}
\SetKwIF{If}{ElseIf}{Else}{if}{then}{else if}{else}{endif}
\SetKw{Pick}{pick}
\SetKw{Return}{return}
\SetNlSkip{2em}
\newenvironment{algorithm-hbox}{\hbadness=10000\begin{algorithm}}{\end{algorithm}}

\newcommand{\nosemic}{\renewcommand{\@endalgocfline}{\relax}}
\newcommand{\dosemic}{\renewcommand{\@endalgocfline}{\algocf@endline}}
\let\oldnl\nl
\newcommand{\nonl}{\renewcommand{\nl}{\let\nl\oldnl}}

\newenvironment{enumeratei}{\begin{enumorig}[label=\textup{(\roman*)}, noitemsep, topsep=1mm, labelindent=.5em, leftmargin=*]}{\end{enumorig}}

\DeclareMathOperator{\up}{UP}

\renewcommand{\leq}{\leqslant}
\renewcommand{\geq}{\geqslant}

\newcommand{\set}[1]{\left\{#1\right\}}
\newcommand{\norm}[1]{{\left|#1\right|}}
\newcommand{\epsi}{\varepsilon}

\newcommand{\Nat}{\mathbb{N}}
\newcommand{\calP}{\mathcal{P}}
\newcommand{\calT}{\mathcal{T}}
\newcommand{\calC}{\mathcal{C}}
\newcommand{\calS}{\mathcal{S}}
\newcommand{\DEF}{\em}  
\newcommand{\undefined}{\textrm{\tt undefined}}

\title{\bf Pathwidth and nonrepetitive list coloring}

\author{Adam G\k{a}gol\\
\small Theoretical Computer Science Department\\[-0.8ex]
\small Faculty of Mathematics and Computer Science\\[-0.8ex]
\small Jagiellonian University\\[-0.8ex] 
\small Krak\'{o}w, Poland\\
\small\tt gagol@tcs.uj.edu.pl\\
\and
Gwena\"{e}l Joret\thanks{G.\ Joret was supported by a  
DECRA Fellowship from the Australian Research Council.} \\
\small Computer Science Department\\[-0.8ex]
\small Universit\'e Libre de Bruxelles\\[-0.8ex] 
\small Brussels, Belgium \\
\small\tt gjoret@ulb.ac.be\\
\and
Jakub Kozik\thanks{J.\ Kozik and P.\ Micek were supported by 
the Polish National Science Center, grant 2011/01/D/ST1/04412.} \qquad  Piotr Micek\footnotemark[2] \\
\small Theoretical Computer Science Department\\[-0.8ex]
\small Faculty of Mathematics and Computer Science\\[-0.8ex]
\small Jagiellonian University\\[-0.8ex] 
\small Krak\'{o}w, Poland\\
\small\tt \{jkozik,micek\}@tcs.uj.edu.pl\\
}

\begin{document}

\maketitle

\begin{abstract}
A vertex coloring of a graph is {\em nonrepetitive} if there is no path in the graph 
whose first half receives the same sequence of colors as the second half. 
While every tree can be nonrepetitively colored with a 
bounded number of colors ($4$ colors is enough), 
Fiorenzi, Ochem, Ossona de Mendez, and Zhu recently showed that this does not 
extend to the list version of the problem, that is, for every $\ell \geq 1$ there 
is a tree that is not nonrepetitively $\ell$-choosable. 
In this paper we prove the following positive result, which complements the result of 
Fiorenzi {\it et al.}: There exists a function $f$ such that every tree of pathwidth $k$ is nonrepetitively $f(k)$-choosable. 
We also show that such a property is specific to trees by constructing a family of pathwidth-$2$ graphs that are not nonrepetitively $\ell$-choosable for any fixed $\ell$. 
\end{abstract}

\section{Introduction}

A \emph{repetition} of length $r$ ($r\geqslant 1$) in a sequence of symbols is a subsequence of consecutive terms of the form $x_1\ldots x_r x_1\ldots x_r$. 
A sequence is \emph{nonrepetitive} (or \emph{square-free}) if it does not contain a repetition of any length. 
In 1906 Thue proved that there exist arbitrarily long nonrepetitive sequences over an alphabet of size $3$ (see \cite{Ber95,Thu06}). 
The method discovered by Thue is constructive and uses substitutions over a given set of symbols.
 
A different approach to creating long nonrepetitive sequences 
was recently introduced by Grytczuk, Kozik, and Micek~\cite{GKM}: 
Generate a sequence by iteratively appending a random symbol at the end, and each time a repetition appears erase the repeated block. 
(For instance, if the sequence generated so far is $abcb$ and we add $c$, then we erase the last two symbols, bringing us back to $abc$.) 
By a simple counting argument one can prove that with positive probability the length of the constructed sequence eventually exceeds any finite bound, provided the alphabet has size at least $4$. 
This is one more than in Thue's result but the proof is more flexible and can be adapted to other settings. For instance, it led to a very short proof that for every $n\geq1$ and every sequence of sets $L_1,\ldots,L_n$, each of size at least $4$, there exists a nonrepetitive sequence $s_1s_2\ldots s_n$ where $s_i\in L_i$ 
for all $i$ (see \cite{GKM}), a theorem first proved 
by Grytczuk, Przyby{\l}o, and Zhu~\cite{GPZ} via an intricate application of the Lefthanded Local Lemma. 
Whether the analogous statement for lists of size $3$ is true remains an exciting open problem. 

In this paper we make use of the above-mentioned approach to color trees nonrepetitively. 
Given an (undirected, simple) graph $G$,  we denote by $V(G)$ and $E(G)$ its vertex set and edge set, respectively. 
A coloring $\phi:V(G)\to\Nat$ of the vertices of $G$ is \emph{nonrepetitive} if there is no repetition in the color sequence of any path in $G$; that is, $\phi$ is nonrepetitive if for every path $P$ with an even number of vertices  
the sequence of colors on the first half of $P$ is distinct from the sequence of 
colors on the second half of $P$. 
(We remark that all paths in this paper are simple, that is, 
contain no repeated vertex.)
The minimum number of colors used in a nonrepetitive coloring of $G$ is called the \emph{Thue chromatic number} of $G$ and is denoted by $\pi(G)$. 
Now, given a graph $G$, suppose that each vertex $v\in V(G)$ has a preassigned list of available colors $L_v \subset \mathbb{N}$. 
A coloring of $G$ with these lists is a coloring $\phi$ of $G$ such that $\phi(v)\in L_v$ for each vertex $v\in V(G)$. 
The \emph{Thue choice number} of $G$, denoted by $\pi_l(G)$, is the minimum $\ell$ such that, for every list assignment $\set{L_v}_{v\in V(G)}$  with $\norm{L_v}\geq \ell$ for each $v\in V(G)$, there is a nonrepetitive coloring of $G$ with these lists. 

Similarly as for many graph coloring parameters, the Thue chromatic (choice) number can be bounded 
from above by a function of the maximum degree: 
Alon, Grytczuk, Ha{\l}uszczak, and Riordan~\cite{AGHR02} proved that for every graph $G$ with maximum degree $\Delta$ we have $\pi(G) \leq \pi_l(G) \leq c \cdot \Delta^2$ for some absolute constant $c$. 
A number of subsequent works~\cite{DJKW, Grytczuk, Gry07b, HJ-DM11} focused on reducing the value of the constant $c$, 
the current best bound being $\pi_l(G) \leq (1+o(1)) \Delta^2$ (see~\cite{DJKW}). 
Alon {\it et al.}~\cite{AGHR02} also showed that there are graphs with maximum degree $\Delta$ with $\pi(G)=\Omega\left(\frac{\Delta^2}{\log\Delta}\right)$.  
(Whether this can be improved by a $\log\Delta$ factor remains an open problem.)

It is not difficult to show that every tree has Thue chromatic number at most $4$ (see \cite{BGKNP07}), which is best possible. 
This result was generalized to graphs of bounded treewidth by 
K{\"u}ndgen and Pelsmajer~\cite{KP08}. 
They proved that $\pi(G)\leq 4^k$ for every graph $G$ of treewidth $k$. 
It is not known whether this upper bound can be improved to a polynomial in $k$. 
However, if one considers graphs of pathwidth $k$ instead, a polynomial bound is known:  
It was shown by Dujmovi{\'c} {\it et al.}~\cite{DJKW} that $\pi(G)\leq 2k^2 +6k +1$ for every graph $G$ of pathwidth $k$. 
(We note that quadratic might not be the right order of magnitude here.) 
 
Probably the most intriguing open problem regarding the Thue chromatic number is whether it is bounded for all planar graphs,  
a question originally asked by Grytczuk~\cite{Gry07b}. A $O(\log n)$ upper bound is known~\cite{DFJW}, and from 
below Ochem constructed a planar graph requiring $11$ colors (see~\cite{DFJW}). 

The main focus of this paper is the list version of the parameter, the Thue choice number. 
As mentioned at the beginning of the introduction, we have $\pi_l(P)\leq 4$ for every path $P$, and it is open whether this bound can be improved to $3$.  
Fiorenzi, Ochem, Ossona de Mendez, and Zhu~\cite{FOOZ11} gave the first example of a class of graphs where the Thue chromatic and Thue choice numbers behave very differently: While trees have Thue chromatic number at most $4$, they showed that the Thue choice number of trees is unbounded. 
Clearly, trees with large Thue choice number must have large maximum degree, and 
in fact one can deduce from the proof in~\cite{FOOZ11} that there are trees with maximum degree $\Delta$ and Thue choice number 
$\Omega(\frac{\log\Delta}{\log \log \Delta})$.  
Kozik and Micek~\cite{KM13} subsequently showed that  
a better-than-quadratic upper bound in terms of the maximum degree exists for trees:  For every $\epsi>0$ there exists $c > 0$ such that $\pi_l(T)\leq c\Delta^{1+\epsi}$ for every tree $T$ of maximum degree $\Delta$. (Bridging the significant gap between the upper and lower bounds remains an open problem.)

Note that graphs of bounded treewidth have unbounded Thue choice number since this is 
already the case for trees. 
On the other hand, Dujmovi{\'c} {\it et al.}~\cite{DJKW} observed that $\pi_l(G)$ is bounded when $G$ is a graph of pathwidth $1$. 
This prompted the authors of~\cite{DJKW} to ask whether $\pi_l(G)$ is bounded more generally when $G$ has bounded pathwidth (which is the case for the Thue chromatic number).  
Also, since connected graphs $G$ of pathwidth $1$ are caterpillars, and thus trees in particular, they also asked the same question but with $G$ moreover required to be a tree.    
A second motivation for the latter question was that the trees with arbitrarily large Thue choice number constructed by Fiorenzi {\it et al.}~\cite{FOOZ11} also have unbounded pathwidth. 

In this paper we answer both questions.
First, we give a simple construction showing that the Thue choice number is unbounded for graphs of bounded pathwidth; in fact, this is true even for graphs of pathwidth $2$ 
(which is best possible as noted above):

\begin{theorem}\label{thm:pw2}
For every $\ell\geq1$, there is a graph $G$ of pathwidth $2$ with $\pi_l(G)\geq \ell$.
\end{theorem}
 
Next, we address the case of trees and prove that their Thue choice number is bounded from above by a function of their pathwidth:

\begin{theorem}\label{thm:tree} 
There is a function $b:\Nat \to \Nat$ such that $\pi_l(T)\leq b(k)$ for every tree $T$ of pathwidth $k$.
\end{theorem}

The proof of Theorem~\ref{thm:tree} combines an induction on the pathwidth with the algorithmic method of Grytczuk {\it et al.}~\cite{GKM} to produce arbitrarily long nonrepetitive sequences described at the beginning of the introduction. 
This method, which finds its roots in the celebrated algorithmic proof of the Local Lemma by Moser and Tardos~\cite{MT10}, was extended to produce nonrepetitive colorings of graphs (in~\cite{DJKW}) and trees (in~\cite{KM13}). 
Part of our proof consists in adapting the ideas from~\cite{DJKW, KM13} to the situation under consideration. 

We note that the bounding function $b(k)$ in Theorem~\ref{thm:tree} stemming from our proof is quite large, it is doubly exponential in $k$. 

The paper is organized as follows: In Section~\ref{sec:definitions} we introduce definitions and terminology. 
Then we prove Theorem~\ref{thm:pw2} in Section~\ref{sec:pw2}, and Theorem~\ref{thm:tree} in Section~\ref{sec:trees}. 

\section{Definitions}
\label{sec:definitions}
For an integer $n\geq1$, we let $[n] := \set{1,\ldots,n}$. 
Also, given two integers $a, b$ with $a \leq b$ we let $[a,b]:=\set{a,a+1,\dots, b}$, which we call an {\DEF interval}. 

Graphs in this paper are finite, simple, and undirected. 
The vertex set and edge set of a graph $G$ are denoted $V(G)$ and $E(G)$, respectively. 
Note that, since only simple graphs are considered, resulting loops and parallel edges are removed 
when contracting edges in a graph.  
A graph $H$ is a {\DEF minor} of a graph $G$ if $H$ can be obtained from a subgraph of $G$ by contracting edges. 

A {\DEF tree decomposition} of a graph $G$ is a pair $(T, \mathcal{C})$ where 
$T$ is a tree and $\mathcal{C}$ is a collection $\{T_v:  v\in V(G)\}$    
of non-empty subtrees of $T$ such that $V(T_u) \cap V(T_v) \neq \emptyset$ for every edge $uv \in E(G)$. 
The {\DEF width} of the tree decomposition $(T, \mathcal{C})$ is the maximum, over every $x\in V(T)$, 
of the number of subtrees in $\mathcal{C}$ containing $x$, minus $1$.  
The {\DEF treewidth} of $G$ is the minimum width of a tree decomposition of $G$. 
{\DEF Path decompositions} and {\DEF pathwidth} are defined analogously with the tree $T$ required instead to be a path. 
Treewidth and pathwidth are minor-closed parameters, in the sense that every minor of a graph $G$ has treewidth 
(pathwidth) at most that of $G$. 
We refer the reader to Diestel's textbook~\cite{Diestel} for an introduction to the theory of treewidth and graph minors. 

The {\DEF length} of a path is the number of its edges.
The {\DEF height} of a rooted tree $T$ is the maximum length of  a path from the root to a leaf of $T$. 
(Thus $T$ has height $0$ if it consists of a unique vertex.)  
The {\DEF height} of a vertex $v$ of $T$ is the length of the path from the root to $v$ in $T$. 

\section{Graphs of pathwidth 2}
\label{sec:pw2}

Let $G_{n,\ell}$ be the graph constructed from the path on $2n$ vertices where every second vertex is blown up to $\binom{\ell n}{\ell}$ vertices forming 
an independent set.
Formally, \[V(G_{n,\ell})=\set{v_{2i-1}\mid i\in[n]} \cup \set{v_{2i}^j\mid i\in[n],\ j\in\left[\binom{\ell n}{\ell}\right]},\]  
and two vertices are adjacent in $G_{n,\ell}$ if and only if their lower indices differ by exactly $1$.
Also, let $V_{i}:= \set{v_{i}}$ for each odd index $i\in [2n]$ and  
$V_{i}:= \set{v_{i}^j\mid j\in\left[\binom{\ell n}{\ell}\right]}$ 
for each even index $i\in [2n]$.   

It is not difficult to check that $G_{n,\ell}$ has pathwidth at most $2$ (with equality for $n \geq 2$ and $\ell\geq1$). 
Thus Theorem~\ref{thm:pw2} follows from the following theorem. 

\begin{theorem}
Let $\ell$ and $n$ be integers such that $\ell\geq1$ and $n >  e^{\ell+2}$. 
Then $\pi_l(G_{n,\ell})> \ell$. 
\end{theorem}
\begin{proof}
Consider the following list assignment for the vertices of $G_{n,\ell}$. 
For each odd index $i=2t+1\in [2n]$, vertex $v_{i}$ is assigned the list 
\[
L_{i} := \{t\ell + 1, t\ell + 2, \dots, t \ell + \ell\}. 
\]
Thus these $n$ lists have size $\ell$, are pairwise disjoint, and their union is $[\ell n]$. 
Next, enumerate the $\ell$-subsets of $[\ell n]$ in an arbitrary way. 
Then, for each even index $i\in[2n]$ and index $j\in \left[\binom{\ell n}{\ell}\right]$, 
vertex $v_{i}^j$ is assigned the list which is the $j$-th set in that enumeration. 

We claim that, because $n$ was chosen to be strictly larger than $e^{\ell+2}$, there cannot be a nonrepetitive coloring of $G_{n,\ell}$ with these lists. 
Arguing by contradiction, let us suppose that $\phi$ is such a coloring. 

With a slight abuse of notation, for $i\in [2n]$ we use the shorthand $\phi(V_{i})$ for the set $\cup_{u \in V_{i}} \set{\phi(u)}$. 
Consider an interval $I\subseteq[2n]$ of the form $I=[a,a+4k+1]$ with $k \geq 0$. 
Suppose that the following two conditions are satisfied: 
\[
\begin{array}{ll}
\phi(v_{i}) \in \phi(V_{i+2k+1}) &  \text{for each } i \in [a,a+2k], i \text{ odd} \\[1ex]
\phi(v_{i+2k+1}) \in \phi(V_{i}) &  \text{for each } i \in [a,a+2k], i \text{ even.} 
\end{array}
\]
Then it is easy to check that there exists a path $w_{a},\ldots, w_{a+4k+1}$ in $G_{n,\ell}$ with $w_j=v_j$ for $j$ odd and $w_j\in V_j$ for $j$ even, for all $j\in I$, such that the color sequence $\phi(w_a),\ldots,\phi(w_{a+4k+1})$ is a repetition (of size $2k+1$).
Since this cannot happen, it follows that there exists an index $i \in [a,a+2k]$ for which one of the above two conditions is not satisfied. 
For every such index $i$, we say that the pair $(p, q)$ is a {\DEF witness} (for interval $I$), where $\{p, q\}=\{i, i+2k+1\}$ with $p$ odd and $q$ even.   

Next, consider an even index $q\in [2n]$. 
Observe that $\left| [\ell n] - \phi(V_{q})\right| \leq \ell -1$, since $\phi(V_{q})$ contains at least one color from each $\ell$-subset of $[\ell n]$. 
Combining this with the fact that vertices $v_{p}$ with odd index $p\in [2n]$ have pairwise disjoint lists, we deduce that there are at most $\ell-1$ odd indices $p\in [2n]$ such that the pair $(p, q)$ is a witness. 
Summing up over every even index $q\in [2n]$, it follows that there are at most 
\[
n (\ell-1)
\]
distinct witnesses in total. 

Now consider a witness $(p, q)$ and let $|p-q| = 2k+1$. 
The pair $(p, q)$ is a witness for at most $2k+1$ intervals $I\subseteq[2n]$ of the form $I=[a,a+4k+1]$. 
Since there are exactly $2n-4k-1$ intervals $I$ of the latter form 
and each interval of that form must have a witness, 
it follows that the number of witnesses $(p, q)$ with $|p-q| = 2k+1$ is at least $\frac{2n-4k-1}{2k+1}$. 
Summing up over every possible value of $k$ (that is, $k=0, 1, \dots, \lfloor (n-1)/2\rfloor$), we obtain that the total number of witnesses is at least 
\begin{align*}
\sum_{k=0}^{\lfloor (n-1)/2\rfloor} \frac{2n-4k-1}{2k+1}
&\geq \sum_{k=0}^{\lfloor (n-1)/2\rfloor} \left(\frac{n}{k+1} -2\right)\\
&\geq n\left(\sum_{k=1}^{\lfloor (n+1)/2\rfloor} \frac{1}{k}\right) - 2\left\lfloor \frac{n+1}{2}\right\rfloor\\
&\geq n\ln\left(\left\lfloor\frac{n+1}{2}\right\rfloor\right)-(n+1)\\
&\geq n\ln\left(\frac{n}{2}\right) -(n+1) \\
&\geq n\ln n -3n.
\end{align*}
It follows that  
\[
 n(\ell-1) \geq n\ln n -3n,
\]
which contradicts the assumption that $n > e^{\ell+2}$.
\end{proof}

\section{Trees of bounded pathwidth}
\label{sec:trees}

A {\DEF path-partition} of a tree $T$ is a pair $(\calT, \calP)$ where $\calT$ is 
a rooted  tree and $\calP$ is a collection $\{ P_x: x\in V(\calT)\}$ of vertex-disjoint paths of $T$ 
which collectively partition the  vertex set of $T$ and 
such that $xy \in E(\calT)$ if and only if there is an edge between a vertex from $P_x$ and a vertex from $P_y$ in $T$. 
Observe that a consequence of this definition is that $\calT$ is a minor of $T$. 
The {\DEF root-path} of $(\calT, \calP)$ is the path $P_{x}$ where $x$ is the root of $\calT$. 
Now, consider a path  $P_{x}$ with $x$ distinct from the root.  
The path $P_{x}$ has a {\DEF center}, defined as the endpoint in $P_x$ of the edge in $T$ linking $P_x$ to 
$P_y$ where $y$ is the parent of $x$ in $\calT$. 
The {\DEF height} of the path-partition $(\calT, \calP)$ is the height of $\calT$. 

When considering a path-partition $(\calT, \calP)$ of a tree $T$, it will be useful 
to embed $T$ itself in the plane in a way that is `faithful' to the path-partition. 
This leads to the following definition: An embedding of $T$ in the plane 
is {\DEF faithful} to the path-partition  $(\calT, \calP)$ if 
each path in $\calP$ is drawn horizontally, and contracting each such path into one of its vertices 
we obtain some plane embedding of $\calT$, with its root drawn at the bottom and its edges going up. 
See Figure~\ref{fig:representation} for an illustration. 
\begin{figure}
\includegraphics[width=0.93\textwidth]{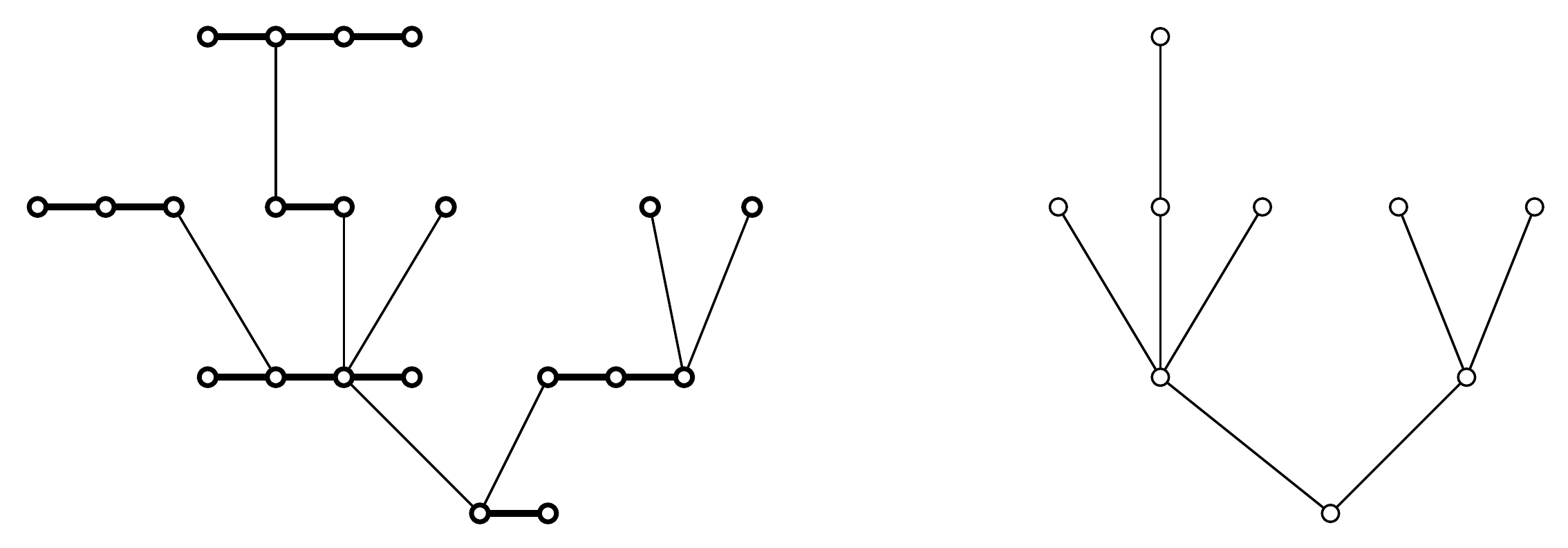}
\caption{\label{fig:representation}Left: A tree $T$ faithfully embedded according to some path-partition $(\calT, \calP)$ (the paths in $\calP$ are drawn in bold). 
Right: The rooted tree $\calT$. 
}
\end{figure}
As the paths in $\calP$ are drawn horizontally, they 
have a natural orientation from left to right. 
Every edge $e$ of $T$ is either {\DEF horizontal} or {\DEF vertical}, depending 
on whether $e$ belongs to some path in $\calP$ or not.  

Our motivation for considering path-partitions is the following lemma. 

\begin{lemma}\label{lemma:representation}
Every tree of pathwidth $k$ has a path-partition of height at most $2k$. 
\end{lemma}

\begin{proof}
We prove the following stronger statement: 
For every tree $T$ of pathwidth $k$ and every vertex $u\in V(T)$, there is a path-partition of $T$ of height at most $2k$ with $u$ in the root-path.  

The proof is by induction on $k$. 
For $k=0$, the tree $T$ consists of the single vertex $u$.
Clearly, it has a path-partition of height $0$ with $u$ in the root-path. 
Now suppose $k>0$ for the inductive case. 
Let $(P,\calC)$ be a path decomposition of $T$ of width $k$, where $T_{v}\in\calC$ denotes the path associated to vertex $v\in V(T)$. 
Enumerate the vertices of the path $P$  indexing the path decomposition as $p_1,\ldots,p_n$, in order.
We may assume without loss of generality that there are (non-necessarily distinct) vertices $x,y\in V(T)$ such that $p_1\in V(T_{x})$ and $p_n\in V(T_{y})$ (otherwise $P$ could be shortened).
Let $Q_1=v_1v_2\ldots v_m$ denote the unique path in $T$ between $x=v_1$ and $u=v_m$ in $T$. 
Let $z$ be the vertex of $Q_1$ that is closest to $y$ in $T$. 
Let $Q_2=w_1w_2\ldots w_{m’}$ denote the unique path in $T$ between $y=w_1$ and $z=w_{m’}$ in $T$.
Notice that $V(T_{v_i})\cap V(T_{v_{i+1}})\neq\emptyset$ for each $i\in\set{1,\ldots,m-1}$ and that each $V(T_{v_i})$ induces a subpath of $P$. 
Similarly, $V(T_{w_i})\cap V(T_{w_{i+1}})\neq\emptyset$ for each $i\in\set{1,\ldots,m’-1}$ and each $V(T_{w_i})$ induces a subpath of $P$. 
Thus we deduce that $\bigcup_{i=1}^m V(T_{v_i}) \cup \bigcup_{i=1}^{m’} V(T_{w_i}) =\set{p_1,\ldots,p_n}$.

Consider the forest $T-(V(Q_1) \cup V(Q_2))$ and let $D_1,\ldots,D_c$ denote its components.
Observe that each tree $D_j$ ($j\in\set{1,\ldots,c}$) has pathwidth at most $k-1$.
Indeed, $(P,\set{T_v\mid v\in V(D_j)})$ is a path decomposition of $D_j$, and for each $i \in\{1, \dots, n\}$  we have $\norm{\set{v\in V(D_j)\mid p_{i}\in V(T_v)}}\leq \norm{\set{v\in V(T)\mid p_{i}\in V(T_v)}}-1\leq k-1$.

For each $j\in\set{1,\ldots,c}$, let $d_j$ denote the unique vertex of $D_j$ having a neighbor in $Q_1 \cup Q_2$ in $T$. 
By induction, each tree $D_j$ ($j\in\set{1,\ldots,c}$) has a path-partition $(\calT_j,\calP_j)$ of height at most $2(k-1)$ such that $d_j$ in the root-path. Let $r_j$ denote the root of $\calT_j$. 

We construct a path-partition $(\calT,\calP)$ of $T$ as follows: 
$\calT$ consists of the disjoint union of $\calT_1, \dots, \calT_c$ plus two extra vertices $q_1$ and $q_2$ with $q_1$ the root of $\calT$ and $q_2$ a child of $q_1$. 
The paths associated to $q_1$ and $q_2$ are $Q_1$ and $Q_2$, respectively. 
For each $j \in \{1,\dots,c\}$, we make $q_1$ or $q_2$ adjacent to $r_j$, depending whether $d_j$ has a neighbor in $Q_1$ or $Q_2$ in $T$.    
It is easy to verify that $(\calT,\calP)$ is a path-partition of $T$ of height at most $2(k-1)+2=2k$.
\end{proof}

The fact that trees of bounded pathwidth have path-partitions of bounded height is a natural observation. 
It is thus likely that this observation was made before though we are not aware of any relevant reference. 
We also note that we made no effort to optimize the bound in Lemma~\ref{lemma:representation} and we do not know whether the factor $2$ is unavoidable. 

Let $T$ be a tree and fix a path-partition $(\calT, \calP)$ of $T$. 
Suppose further that $T$ is embedded in the plane faithfully to $(\calT, \calP)$.  
We use the following terminology when discussing paths in $T$. 
First, every path $P\in \calP$ has a corresponding {\DEF level}, which is defined as the height 
of the corresponding vertex in $\calT$. By extension, every vertex of $T$ has a level, 
the level of the path in $\calP$ it belongs to. 
Define the {\DEF base} of an arbitrary path $P$ in $T$ as the subpath induced by the vertices of $P$ of minimum level.    
Since the base of $P$ is a subpath of a path in $\calP$, its vertices are ordered from left to right by the plane embedding of $T$. 
The path $P$ is said to be {\DEF ascending} if at least one of its two endpoints belongs to its base. 
(We note that in particular all paths in $\calP$ are ascending, even though they are drawn horizontally in the embedding of $T$.) 
Each ascending path $P$ in $T$ has a {\DEF source}, defined as  
the endpoint of $P$ that is in the base of $P$; 
in case both endpoints are in the base, the left-most one is selected as the source.  
We typically think of ascending paths $P$ as being directed from their source to their other endpoint so that 
the notion of $i$th vertex of $P$ is well defined, the first vertex being the source.  
An ascending path $P$ with at least two vertices either {\DEF goes right} or {\DEF goes left} or {\DEF goes up}, depending 
on whether the second vertex of $P$ is on the base and to the right of the source, or 
on the base and to the left of the source, or is one level higher. 

Next we generalize the notion of repetition as follows. 
A {\DEF near repetition} is a sequence of the form $x_1\ldots x_ry_1\ldots y_g x_1\ldots x_r$, 
where $r\geqslant 1$ is its {\DEF length}, and $g\geq 0$ is said to be 
its {\DEF gap}. 
(Thus for $g=0$ this is the usual notion of repetition.) 
Now, let us return to our tree $T$ from the previous paragraph, 
and let $\phi$ denote an arbitrary coloring of its vertices. 
A slightly technical but key definition for our purposes is the following:
An ascending path $P$ of $T$ is said to be {\DEF $\phi$-bad} 
if, enumerating its vertices as $v_{1}, v_{2}, \dots, v_{p}$ starting from its source, 
the sequence $\phi(v_{1})\phi(v_{2}) \dots \phi(v_{p})$ forms
a near repetition $x_1\ldots x_ry_1\ldots y_g x_1\ldots x_r$ of length $r$ and gap $g$ 
where at most $r$ vertices from $v_{r+1}, \dots, v_{r+g}$ lie in the base of $P$. 
(That is, either $g \leq r$, 
or $g > r$ but at most $r$ vertices from the `gap' section are in the base of $P$.) 
An ascending path that is not $\phi$-bad is said to be {\DEF $\phi$-good}.  
We sometimes drop $\phi$ when using these two adjectives if 
the coloring $\phi$ they refer to is clear from the context

Equipped with these definitions we may now state the following technical lemma, 
which turns out to be the heart of the proof. 

\begin{lemma}\label{lemma:good}
There is a function $f:\Nat \times \Nat \to \Nat$ such that, for every 
$\ell \geq 1$, every $h \geq 0$, and every tree $T$ faithfully embedded according to a path-partition $(\calT, \calP)$ of $T$ of height $h$ with lists $L_v$ $(v\in V(T))$ of colors of size $f(\ell,h)$, one can find sublists $S_v \subseteq L_v$ $(v\in V(T))$ of size $\ell$ such that, for every coloring $\phi$ of $T$ with these sublists, every ascending path of $T$ is $\phi$-good. 
\end{lemma}

In order to motivate Lemma~\ref{lemma:good}, we show that with only a little extra effort (greedy coloring from the sublists) it implies Theorem~\ref{thm:tree}.

\begin{proof}[Proof of Theorem~\ref{thm:tree} (assuming Lemma~\ref{lemma:good})]
Let $T$ be a tree of pathwidth $k$ and let $L_v$ $(v\in V(T))$ be a list assignment for the vertices of $T$ where each list has size $b(k):=f(2k+1, 2k)$, where $f$ is the function from Lemma~\ref{lemma:good}.
By Lemma~\ref{lemma:representation} there is a path-partition $(\calT,\calP)$ of $T$ of height at most $2k$.
By Lemma~\ref{lemma:good} there are sublists $\set{S_v}_{v\in V(T)}$ with $S_v\subseteq L_v$ and $\norm{S_v}=2k+1$ for each vertex $v\in V(T)$, such that
in any coloring $\phi$ of $T$ with these sublists, 
all ascending paths in $T$ are $\phi$-good.

We define a nonrepetitive coloring $\phi$ of $T$ with the lists $S_v$ ($v\in V(T)$) in a greedy manner. 
We color the vertices of $T$ one by one in non-decreasing order of their levels. 
Let $v\in V(T)$ be a vertex under consideration. 
Let $v_{1}v_{2}\dots v_{p}$ denote the shortest path in $T$ from $v_{1}=v$ to the root-path 
(thus it enters the root-path in vertex $v_p$). 
Recall that every edge of the form $v_{i-1}v_{i}$ with $i\in\set{2,\ldots,p}$ is either horizontal or vertical in $T$.
Let $$G(v):=\big\{v_i\mid i\in\set{2,\ldots,p} \textrm{ and } v_{i-1}v_i \textrm{ is a vertical edge}\big\}.$$ 
Each vertex in $G(v)$ is said to be a {\DEF guard} for vertex $v$. 
Note that $\norm{G(v)}$ is exactly the level of vertex $v$ in $T$, and thus 
in particular $\norm{G(v)}\leq 2k$. 
We color $v$ as follows: 
Let $\phi(v)$ be an arbitrarily chosen color from 
the non-empty set $S(v)-\phi(G(v))$. 
(Here, $\phi(G(v))$ denotes the set of colors used for vertices in $G(v)$; note that these 
vertices are already colored since they lie on lower levels.) 

We claim that $\phi$ is a nonrepetitive coloring of $T$.
Arguing by contradiction, suppose that there is a repetitively colored
path $P=v_1\ldots v_p w_1 \ldots w_p$. 
Consider the edge $e=v_p w_1$. 
First we show that $e$ belongs to the base of $P$. 
Suppose not, and consider the shortest subpath of $P$ that includes the edge $e$ and has one endpoint in the base of $P$. 
Reversing $P$ if necessary, we may assume without loss of generality that this subpath 
is of the form $v_pw_1\ldots w_m$, with $w_m$ being the only vertex on the base. 
Observe that $w_{m-1}w_m$ is a vertical edge.
This implies that $w_m$ is a guard for all the vertices in $\set{v_1,\ldots,v_p, w_1,\ldots,w_{m-1}}$. 
In particular, the color $\phi(w_m)$ cannot have been used for vertex $v_m$ 
since $\phi(v_m) \in S(v_{m})-\phi(G(v_{m}))$, contradicting the fact that $\phi(w_m)=\phi(v_m)$. 
Therefore, the edge $e$ must lie in the base of $P$.

Let $\ell$ and $r$ be the number of vertices in  $\set{v_1,\ldots,v_p}$ 
and $\set{w_1,\ldots,w_p}$, respectively, that are in the base of $P$. 
Reversing $P$ if necessary,  
we may assume without loss of generality that $\ell \leq r$.
Consider the path $P'=w_rw_{r-1}\ldots w_1 v_p \ldots v_1$. 
Observe that $P'$ is an ascending path as one of its endpoints, namely $w_r$, is in the base of $P'$.
Now, $\phi(w_r)\phi(w_{r-1})\ldots \phi(w_1) \phi(v_{p}) \ldots \phi(v_1)$ is a near repetition of length $r$ with gap $p-r$, and exactly $\ell$ vertices from the gap section are in the base of $P'$. 
Since $\ell\leq r$, we deduce that $P'$ is $\phi$-bad, contradicting the fact 
that every ascending path is $\phi$-good. 
\end{proof}

An {\DEF arborescence} is a rooted directed tree where the edges are directed 
away from the root. It will be convenient to consider arborescences that 
are embedded in the plane without edge crossings in such a way that 
the root is drawn at the bottom and all arcs go up (thus the source of an arc is drawn below its sink), which we simply call {\DEF plane} arborescences. 
The {\DEF height} of a vertex in an arborescence is defined as its distance to the root, 
thus in particular the root has height $0$. 
The {\DEF rightmost path} of a plane arborescence is the 
path obtained by starting from the root and always taking the rightmost 
arc going up, until reaching a leaf. 

We classify directed paths in a plane arborescence $A$ 
as being good or bad w.r.t.\ a given coloring $\phi$ of $A$, 
similarly as for ascending paths: Say that 
a directed path $P$ is {\DEF $\phi$-bad} if, enumerating its vertices as $v_{1}v_{2} \dots v_{p}$ in order, 
the sequence $\phi(v_{1})\phi(v_{2}) \dots \phi(v_{p})$ can be written 
as a near repetition $x_1\ldots x_ry_1\ldots y_g x_1\ldots x_r$ of length $r$ and gap $g$ 
where at most $r$ vertices from $v_{r+1}, \dots, v_{r+g}$ lie on the rightmost path of $A$. 
(That is, either $g \leq r$, or $g > r$ but at most $r$ vertices from the `gap' section are on the rightmost path.) If the directed path $P$ is not $\phi$-bad then it is {\DEF $\phi$-good}. 

\begin{lemma}\label{lemma:more-technical}
Let $\ell \geq 1$, let $A$ be a plane arborescence, and let $L_v$ $(v\in V(A))$ be lists of colors of size $32\ell^3+1$.
Then one can find sublists $S_v \subseteq L_v$ $(v\in V(A))$ of size $\ell$ such that, for every coloring $\phi$ of $A$ with these sublists, 
every directed path starting on the rightmost path is $\phi$-good. 
\end{lemma}

The interest of Lemma~\ref{lemma:more-technical} is that 
Lemma~\ref{lemma:good} can be proved by iterated applications of 
Lemma~\ref{lemma:more-technical}, as we now show.

\begin{proof}[Proof of Lemma~\ref{lemma:good} (assuming Lemma~\ref{lemma:more-technical})]
The function $f(\ell,h)$ that will be used is defined inductively on $h$ as follows: 
$f(\ell,0) := 32\ell^3+1$, and $f(\ell,h) := f(32(32\ell^3+1)^3+1,h-1)$ for $h>0$.

Let $T$ be a tree with a path-partition $(\calT,\calP)$ of height $h$ and let $L_v$ $(v\in V(T))$ be lists of colors of size $f(\ell,h)$. 
Suppose further that $T$ is faithfully embedded according to $(\calT,\calP)$.  
We prove the lemma by induction on $h$. 
For the base case of the induction, $h=0$, we observe that $T$ is then a path and all ascending paths in $T$ are simply subpaths of $T$.
As $f(\ell,0) = 32\ell^3+1$, by Lemma~\ref{lemma:more-technical} there are sublists $S_v\subseteq L_v$ for each vertex $v\in V(T)$ with $\norm{S_v}=\ell$ such that, for every coloring $\phi$ of $T$ with these 
sublists,  
all ascending paths of $T$ are $\phi$-good, as required.
(As expected, when applying Lemma~\ref{lemma:more-technical} we first turn the path $T$ into an arborescence by directing it from left to right.) 

For the inductive case $h>0$, let $x$ be the root of $\calT$ and let $P_x$ be the root-path of $T$. Let also $D_1,\ldots,D_c$ be the components of the forest $T-V(P_x)$. 
(Note that there is at least one component.)  
For each $i\in\set{1,\ldots,c}$, the path-partition $(\calT,\calP)$ induces 
in a natural way a path-partition $(\calT_i,\calP_i)$ of $D_{i}$ of height at most $h-1$, 
with $\calT_i$ rooted at the only vertex that is a neighbor of $x$ in $\calT$. 
Since $f(\ell,h) = f(32(32\ell^3+1)^3+1,h-1)$, applying induction on $D_{i}$ we obtain for each vertex $v\in V(D_i)$ a sublist $S'_v \subseteq L_v$ of size $32(32\ell^3+1)^3+1$ such that, for every coloring $\phi$ of $D_i$ with these sublists, every ascending path of $D_i$ is $\phi$-good.

Next, for each vertex $v\in V(P_x)$ let $S'_v$ be an arbitrary subset of $L_v$ of size $32(32\ell^3+1)^3+1$. 
Thus, every vertex $v$ of $T$ now has a corresponding sublist $S'_v\subseteq L_v$ of size $32(32\ell^3+1)^3+1$.   
Moreover, given any coloring $\phi$ of the tree $T$ with these sublists, the only ascending paths that could possibly be $\phi$-bad are those having their sources in $P_x$. 
We shall refer to these ascending paths as the {\DEF risky paths} of $T$.   

Enumerate the vertices of the root-path $P_x$ as $v_1v_2\ldots v_{n}$, from left to right.
Define two plane arborescences $A$ and $A'$ from $T$ by rooting $T$ at $v_1$ and $v_n$, respectively, and ensuring that $P_x$ is a prefix of the rightmost path in both instances. 
Note that the rightmost path of $A$ could extend beyond $P_x$ (in case $v_n$ is not a leaf of $T$), and the same is true for the rightmost path of $A'$ (if $v_1$ is not a leaf). 
What is important for our purposes is to observe that each risky path of $T$ starts on the rightmost path in both $A$ and $A'$.
Observe also that each risky path of $T$ that goes right (left) is a directed path in $A$ (respectively $A'$), 
and risky paths that go up are directed in both $A$ and $A'$. 

First, apply Lemma~\ref{lemma:more-technical} on $A$ with list assignment $S'_v$ ($v\in V(A)$), 
giving for each vertex $v\in V(T)$ a sublist $S_v''\subseteq S_v'\subseteq L_v$ of size $32\ell^3+1$.
Next, apply Lemma~\ref{lemma:more-technical} on $A'$ with list assignment $S''_v$ ($v\in V(A)$), 
giving for each vertex $v\in V(T)$ a sublist $S_v \subseteq S_v''\subseteq S_v'\subseteq L_v$ of size $\ell$. 
Since every risky path of $T$ is mapped to a directed path starting on the rightmost path in $A$ or $A'$, 
by the properties of the sublists $S_v''$ and $S_v$ ($v\in V(T)$) 
guaranteed by Lemma~\ref{lemma:more-technical} 
we know that, for every coloring $\phi$ of $T$ with the lists $S_v$ ($v\in V(T)$), 
all risky paths of $T$ are $\phi$-good. 
Therefore, the lists $S_v$ ($v\in V(T)$) have the desired properties. 
\end{proof}

It remains to prove Lemma~\ref{lemma:more-technical}.  
As alluded to in the introduction, we will do so by adapting the algorithmic method used in~\cite{DJKW, GKM, KM13}.

\begin{proof}[Proof of Lemma~\ref{lemma:more-technical}]
Let $N:=32\ell^3+1$ denote the size of the lists. 
For $v\in V(A)$, let $\up(v)$ denote the set of vertices $w \in V(A)$ that can be reached via 
a directed path from $v$ in $A$. (Note that $v \in \up(v)$.) 
In the proof, we will often abbreviate `subset of size $k$' and `sublist of size $k$' into {\DEF `$k$-subset'} and {\DEF `$k$-sublist'}, respectively.  

We define a simple randomized algorithm, Algorithm~\ref{algo:1}, that tries to find an $\ell$-sublist $S_v$ of $L_v$ for each vertex $v\in V(A)$ such that, for every coloring $\phi$ of $A$ with these sublists, every directed path starting on the rightmost path of $A$ is $\phi$-good.
The following informal description of the algorithm is complemented by the more formal description given in Algorithm~\ref{algo:1}.   
The algorithm explores the arborescence $A$ via a depth-first, left-to-right search starting from the root.
The algorithm maintains at all time $\ell$-sublists $S_v\subseteq L_v$  for all vertices $v$ encountered {\em before} 
the current vertex $u$ in the depth-first search of $A$. 
These sublists have the following property: 
For every coloring $\phi$ of these vertices with these sublists ($\phi$ being thus a partial coloring of $A$), 
every directed path starting on the rightmost path of $A$ that is fully colored is $\phi$-good. 
We say that such a partial sublist assignment is {\DEF valid}.

Next, the algorithm treats the current vertex $u$ and tries to maintain the above property. 
To do so, the algorithm first chooses an $\ell$-sublist $S_u\subseteq L_u$  uniformly at random. 
If this new sublist $S_u$ triggers the existence of a $\phi$-bad path in $A$ for some (partial) coloring $\phi$ with the current sublists---that is, the current sublist assignment is no longer valid---it erases some of these sublists as follows: 
Say $v_1\ldots v_{2r+g}$ with $v_{2r+g}=u$ is a $\phi$-bad path with color sequence $\phi(v_1)\ldots\phi(v_{2r+g})$ of the form $x_1\ldots x_r y_1\ldots y_g x_1\ldots x_r$. 
The algorithm then erases the choice for the list $S_v$ for all vertices $v$ 
contributing to the second occurrence of the repeated sequence and their descendants, 
that is, for all $v\in \up(v_{r+g+1})$. 
At the next iteration, $v_{r+g+1}$ becomes the new current vertex, that is, the next vertex to be treated.  
Notice that this makes the algorithm backtrack a number of steps w.r.t.\ the depth-first left-to-right search of $A$.

If on the other hand, the new sublist $S_u$ does not trigger any such bad configuration, then the current sublist assignment remains valid.  
In this case, before proceeding to the next random choice   
the algorithm first tries to extend the current sublist assignment deterministically as much as 
possible. 
(While it might not be clear at first glance why 
this deterministic extension step is needed, we remark that it is actually a key feature of the algorithm without which we could not do the analysis below.) 
This is done as follows:
The algorithm considers the children $u_1,\ldots,u_k$ of $u$ one by one in left-to-right order, 
until a {\DEF problematic} child is identified: 
When considering $u_{j}$, the algorithm checks whether there exist $\ell$-subsets $S_v\subseteq L_v$ for all $v\in \up(u_j)$ such that, taken together, they extend the current sublist assignment in such a way that it remains valid. 
If these subsets exist, the current sublist assignment is extended in this way to the whole subtree rooted at $u_{j}$, 
and the algorithm considers the next child of $u$. 
(If there are more than one valid choice for these sublists, the algorithm chooses one according to a deterministic rule.) 
If no such extension of the current sublist assignment can be found for vertices in $\up(u_j)$, then $u_{j}$ is identified as being a problematic child of $u$, and $u_{j}$ becomes the next vertex to be treated. 
Observe that this effectively makes the algorithm proceed with the depth-first left-to-right search of $A$ for some number of steps.

Let us make some observations concerning the algorithm: Right at the beginning, after selecting a sublist for the root of $A$, two situations can occur: 
(1) No child of the root is problematic. 
Thus a valid sublist assignment for all vertices of $A$ has been found, and the lemma is proved. 
(2) Some child of the root is problematic. 
In this case, it is important to observe that later on each vertex $u$ for which a random sublist $S_u$ is chosen was problematic when its parent was considered.  
It follows in turn that some child of $u$ will be problematic, since otherwise we could extend the sublist assignment to the whole subtree rooted at $u$. 

To summarize, we may assume that we are in case (2) at the beginning, since otherwise we are done. 
This implies that the vertex $u$ that is currently being treated by the algorithm always has a problematic child. 
Moreover, the algorithm {\em will never stop}, simply because while it can erase the choices of 
sublists for some vertices of $A$ it cannot do so for the root, as is easily checked. 
Our proof will then proceed in the following way: 
We run the algorithm until it made $M$ random choices of sublists and then stop it, where $M$ will be some large number 
which is a function of $|V(A)|$ and $\ell$. 
We then carefully set up a concise description (called {\DEF log}) of its execution that is precise enough to allow us to recover from it all random choices that were made by the algorithm. 
Finally, we  count the number of distinct logs that can occur after $M$ random choices, 
and show that, for sufficiently large $M$,  this number is strictly less than ${N \choose \ell}^{M}$.  
From this we deduce that not all sequences of $M$ random choices of sublists can occur in case (2).  
In other words, there is a choice for the sublist of the root of $A$ leaving us in case (1), which then finishes the proof. 

This concludes our informal description of the algorithm,  see Algorithm~\ref{algo:1} for the pseudo code. 
A few remarks about the latter are in order: 
First, we assume that the $\ell$-subsets of $L_v$ have been enumerated for each $v \in V(A)$, so that 
the $j$-th  $\ell$-subset of $L_v$ is well defined for $j\in \left[\binom{N}{\ell}\right]$. 
This ordering also induces an ordering on every subcollection of the collection of $\ell$-subsets of $L_v$.
We also use this enumeration in the proof. 
Second, for simplicity we model the random choices made by the algorithm by a sequence $r_1, r_2, \dots, r_M$ of numbers given in input, each between $1$ and ${N \choose \ell}$, where $r_i$ will be the number used for the $i$-th random choice.  
We call this sequence the {\DEF random input}. 
Third, in line~\ref{algo:choice-bad-path}, the $\phi$-bad path is chosen according to some fixed  rule. 
Similarly, in line~\ref{algo:choice-extension}, the sublists $S_v$ are chosen according to some fixed rule. 
(In each case, the actual rule is irrelevant, as long as it is deterministic.)

\begin{algorithm-hbox}[t]
\caption{Attempts to find sublists $S_v$ of $L_v$ for all $v\in V(A)$, each of size $\ell$, such that for every coloring $\phi$ of $A$ with these sublists, every directed path starting on the rightmost path of $A$ is $\phi$-good.}\label{algo:1}
\SetAlgoLined
\textbf{input:} Lists $L_v$ for all $v\in V(A)$ and random input $r_1, r_2, \dots, r_M$ \\
$i\gets 1$ \\
$u\gets \text{root of A}$\\ 
$S_v \gets \undefined$ for each $v \in V(A)$\\
\While{$i \leq M$} {
  $S_u\gets$ $r_i$-th subset of size $\ell$ of $L_u$\label{algo:def-Sj}\\
  \If{\upshape{there is a $\phi$-bad path starting on the rightmost path of $A$ for some coloring $\phi$ with the lists $S_v$ $(v\in V(A))$}} {
    let $v_1\ldots v_{2r+g}$ with $v_{2r+g}=u$ be a $\phi$-bad path and let \label{algo:choice-bad-path} \\
    \nonl\quad $\phi(v_1)\ldots\phi(v_{2r+g})$  be a sequence of the form $x_1\ldots x_ry_1\ldots y_g x_1\ldots x_r$ \label{algo:def-of-g-and-k}\\
    $S_{v} \gets \undefined$ for each $v \in \up(v_{r+g+1})$ \label{algo:erase-Sj}\\  
    $u \gets v_{r+g+1}$\label{algo:new-current}\\
   }
   \Else{
    \upshape{let $u_1,\ldots,u_k$ denote the children of $u$ ordered from left to right}\label{algo:choice-extension-begin}\\
    $j\gets 1$ \\
    \While{ \upshape{$j\leq k$ and there is a valid extension of the current sublist 
    \quad \quad assignment by some $S_v\subseteq L_v$ for all $v\in \up(u_j)$}}{
     \upshape{choose such sublists $S_v\subseteq L_v$ for all $v\in \up(u_j)$}\label{algo:choice-extension} \\
     $j\gets j+1$ \label{algo:choice-extension-end}
    }
    \If{\upshape{$j=k+1$}}{
    \nonl({\it in this case $i=1$ and $u$ is the root of $A$})\\
    \textbf{return} $S_v$ for all $v\in V(A)$
    }   
    \lElse{\upshape{$u \gets u_{j}$\label{algo:traverse-A}
    }}  
   }
     $i\gets i+1$ \label{algo:i-inc}\\
}
\textbf{report failure}
\end{algorithm-hbox}

In the following, by the $i$-th iteration of the algorithm, we mean the $i$-th iteration of the \texttt{while} loop.
We call operations in lines~\ref{algo:def-of-g-and-k}-\ref{algo:new-current} a {\DEF retraction} of the near repetition $x_1\ldots x_ry_1\ldots y_g x_1\ldots x_r$. 
With a slight abuse of terminology, we will also say that the corresponding $\phi$-bad path has been {\DEF retracted}. 

From now on we argue by contradiction and suppose that the desired sublists for the vertices of $A$ do not exist. 
In other words, we assume that every choice for the sublist of the root at the beginning of the algorithm leaves us in case (2) described above. In particular, for all $M$ and all random inputs $r_1, r_2, \dots, r_M$, Algorithm~\ref{algo:1} 
runs for $M$ steps and then reports failure. 
 
We are going to create a concise description of what Algorithm~\ref{algo:1} does during the $M$ steps of its execution. 
This description is completely determined by the lists and the random input. 
We see the lists $L_v$ ($v\in V(A)$) as being fixed and thus treat the description as a function of the random input $r_1, r_2, \dots, r_M$. 
The description, which we call an {\DEF $M$-log}, consists of a $4$-tuple $(D,\calS,B,\Gamma)$ defined as follows:
\begin{enumeratei}
\item $D=(d_1,\ldots, d_M)$ and  $d_i$ ($i \in [M]$) is the height of the vertex $u$ in Algorithm~\ref{algo:1} at the {\em end} of iteration $i$, when reaching line~\ref{algo:i-inc} after $u$ was updated 
in the {\tt if}--{\tt else} block.   
(Thus, for $i<M$, $d_i$ is simply the height of vertex $u$ at the {\em beginning} of iteration $i+1$.)  
\\
\item 
$\calS:V(A) \to \left[\binom{N}{\ell}\right] \cup \set{\undefined}$ is a function encoding the final partial sublist assignment to vertices of $A$ at the end of iteration $M$. 
More precisely, for each $v\in V(A)$ we have $\calS(v)=j$ if sublist $S_v$ is defined at that moment and $S_v$ is the $j$-th $\ell$-subset of  $L_v$, and $\calS(v)=\undefined$ if $S_v$ is not defined. \\
\item 
$B=(b_1, \ldots, b_M)$ and $b_i$ ($i\in [M]$) are defined as follows: If no bad path was retracted during the $i$-th iteration then $b_i=0$. 
Otherwise, $b_i$ is the number of vertices in the prefix $v_1 \ldots v_{r+g}$ of the retracted bad path $v_1 \ldots v_{2r+g}$ that are on the rightmost path of $A$. 
Observe that $b_i \leq 2r$ in the latter case since at most $r$ vertices from the gap section of a bad path lie on the rightmost path of $A$.\\
\item 
$\Gamma=(\Gamma_1,\ldots,\Gamma_M)$, where $\Gamma_i$ ($i\in[M]$) is defined as follows.  
If no bad path was retracted during the $i$-th iteration then $\Gamma_i=\undefined$. 
Otherwise, letting $v_1 \ldots v_{2r+g}$ be the retracted bad path, we set $\Gamma_i=(\gamma_1, \dots, \gamma_r)$ with $\gamma_j$ ($j\in[r]$) defined as follows: 
Let $\mathcal{X}_j$ denote the collection of $\ell$-subsets of $L_{v_{r+g+j}}$ that have a {\em non-empty intersection with $S_{v_{j}}$}. 
Then $\gamma_j$ is the index of the set $S_{v_{r+g+j}}$ in the collection $\mathcal{X}_j$.  
\end{enumeratei}

Now, our aim is to bound from above the number of distinct $M$-logs $(D,\calS,B,\Gamma)$ by a relatively small function of  $M, |V(A)|$, and $\ell$. 
Recall that the lists $L_v$ $(v\in V(A))$ are fixed, thus $\ell$ and $|V(A)|$ are fixed, and only $M$ and the random input $r_1, \dots, r_M$ vary. 
There are exactly $\binom{N}{\ell}^M$ distinct random inputs of length $M$, and our goal in the following analysis is to deduce that there are $o\left( \binom{N}{\ell}^M \right)$ distinct $M$-logs. (The asymptotic notation is to be interpreted with respect to the variable $M$ of course.)  
This is then a contradiction for $M$ large enough, as  mentioned earlier. 

We start by estimating the number of $M$-tuples  $D=(d_1,\ldots,d_M)$. 
Each sequence $D=(d_1, \ldots, d_M)$ can be injectively mapped to its sequence of differences $(d_2-d_1,\ldots, d_M-d_{M-1})$. 
(Note that $d_1=1$.)
All numbers in this new sequence belong to the set $\set{1,0,-1,-2,\ldots}$, as is easily seen. 
Next we transform that sequence into yet another sequence by replacing each number $k$ by $1$ followed by 
$1-k$ consecutive occurrences of $-1$. 
For instance, the sequence of differences $(1,1,1,1, 1, -2, -1,1)$ gets mapped to $(1,1,1,1, 1, 1, -1,-1,-1,1, -1, -1,1)$.
It is easy to see that the second transformation is also injective. 
The resulting sequence $D'$ is a sequence over the alphabet $\{-1,1\}$.
The number of $1$'s in $D'$ corresponds to the number of times the algorithm assigns a value to some variable $S_{u}$ in line~\ref{algo:def-Sj}, and is thus equal to the number of iterations, that is, $M$.  
The number of $-1$'s in $D'$ is the sum of all values of $r$ over all 
bad paths $v_1\ldots v_{2r+g}$ considered in lines~\ref{algo:def-of-g-and-k}-\ref{algo:erase-Sj}
during the execution.
One can see this as the number of times the algorithm `erases' a value of $S_{v}$ for some $v\in V(A)$ that was set earlier using the random input (note that an execution of line~\ref{algo:erase-Sj} erases $r$ such values).
Thus, this number is at most the total number of executions of line~\ref{algo:def-Sj}, that is, the number of $1$'s in $D'$, which is $M$. 
Hence, $D'$ has size between $M$ and $2M$, and there are at most 
$2^M + 2^{M+1} + \cdots + 2^{2M} \leq 2^{2M + 1}$ such sequences $D'$.

Next we bound the number of different functions $\calS$. 
Note that this number depends only on $N$, $\ell$, and $|V(A)|$, 
so it can be treated as a constant w.r.t.\ $M$.
We denote this number by $c$ (its exact value being irrelevant for the analysis).

Now we turn our attention to the number of possible tuples $B=(b_1,\ldots,b_M)$ in an $M$-log $(D,\calS,B,\Gamma)$. 
Recall that if $b_i \neq 0$ then $b_i \leq 2r$, where $r$ is the size of the repeated part in the near repetition retracted during the $i$-th iteration. 
Hence, the sum of the $b_i$'s is at most twice the total number of $-1$'s in $D'$. 
This implies that $b=\sum_{i=1}^M b_{i} \leq 2M$.
The number of such sequences is easily seen to be at most $2^{3M}$, as one can encode each such sequence as a binary word $1^{b_1}01^{b_2}01^{b_3}0 \ldots1^{b_M}01^{2M-b}$ of length $3M$ (the $i$-th section $1^{b_i}0$ encodes the number $b_i$ for $i=1,\dots, M$, and the padding $1^{2M-b}$ at the end of the word ensures that the length is $3M$).

It remains to estimate the number of possible tuples $\Gamma=(\Gamma_1,\ldots,\Gamma_M)$ which can occur for fixed $D$, $\calS$, and $B$. 
Consider an index $i\in [M]$. 
If $\Gamma_i \neq \undefined$, then each number in the sequence $\Gamma_i$ is the index of some $\ell$-subset of some list $L_v$ ($v\in V(A)$) among those that intersect some other list $S$ of size $\ell$. 
Clearly, the number of $\ell$-subsets of $L_v$ intersecting $S$ is at most $\binom{\ell}{1}  \cdot\binom{N-1}{\ell-1} + \binom{\ell}{2} \cdot\binom{N-2}{\ell-2} + \cdots + 
 \binom{\ell}{\ell-1} \cdot\binom{N-(\ell-1)}{1} + \binom{\ell}{\ell} \cdot\binom{N-\ell}{0}$. 
Since $N$ is much bigger than $\ell$ (recall that $N > 32\ell^3$), each term in the previous sum is upper bounded by the first term $\ell \cdot\binom{N-1}{\ell-1}$, and we obtain that the sum is at most $\ell^2 \cdot\binom{N-1}{\ell-1}$. 
(Of course, this is a rather crude upper bound but it is good enough for our purposes.)    
Hence all numbers in the sequence $\Gamma_i$ are between $1$ and $\ell^2\cdot\binom{N-1}{\ell-1}$. 
Note also that the length of the sequence $\Gamma_i$ is exactly the length of the near repetition retracted during the $i$-th iteration. 
Hence, given $D$ we know exactly which $\Gamma_i$ are defined and what are their lengths.
The sum of these lengths is the total number of $-1$'s in $D'$, which is at most $M$. 
Therefore, for a fixed $D$ there can be at most $\left(\ell^2\cdot\binom{N-1}{\ell-1}\right)^M$ distinct sequences $\Gamma$.

Putting all the previous observations together, we deduce that the number of distinct tuples  $(D,\calS,B,\Gamma)$  is at most 
\begin{align*}
&2^{2M + 1} \cdot c \cdot 2^{3M} \cdot \left(\ell^2 \binom{N-1}{\ell-1}\right)^M \\
&\qquad=O\left( 32^M \cdot \left( \frac{\ell^3}{N}\binom{N}{\ell} \right)^M \right)\\
&\qquad=o\left( \binom{N}{\ell}^M \right), 
\end{align*}
as desired. (The $o(\cdot)$ follows from the fact that $N > 32\ell^3$.) 
This shows that, if $M$ is sufficiently large, 
then the number of possible $M$-logs is strictly smaller than $\binom{N}{\ell}^M$, the number of 
random inputs of length $M$.  
To obtain the desired contradiction, it remains to show that runs of the algorithm on 
different sources produce distinct $M$-logs, 
that is, that any $M$-log $(D,\calS,B,\Gamma)$ uniquely determines 
the random input used by the algorithm to produce it.  
This is exactly what we show next. 

Consider an $M$-log $(D,\calS,B,\Gamma)$ and let $r_1, \dots, r_M$ be {\em any} random input that can lead to its production.  
We prove that $r_1, \dots, r_M$ are uniquely determined by induction on $M$. 
This is clearly true if $M=1$, since the function $\calS$ tells us explicitly which sublist was chosen 
for the root of $A$. 
So assume $M > 1$ for the inductive case. 
Let $D= (d_1,\ldots, d_M)$, $B=(b_1, \ldots, b_M)$, and $\Gamma=(\Gamma_1,\ldots,\Gamma_M)$. 

First suppose that $d_M=d_{M-1}+1$. 
Then no near repetition was retracted during the $M$-th iteration. 
The vertex $u$ that was the current vertex at the beginning of the $M$-th iteration is determined by the 
function $\calS$: 
It is the {\em last} vertex $w\in V(A)$ in the depth-first left-to-right search order from the root 
such that $\calS(w) \neq \undefined$ that has a child $w'$ with $\calS(w') = \undefined$.
(Note that the first such child $w'$ is the problematic child $u_{j}$ of $u$ identified 
when exiting the inner while-loop.)  
Now, observe that $r_M$ is simply the index of $\calS(u)$ among $\ell$-subsets of $L_u$, and is thus completely determined by our $M$-log $(D,\calS,B,\Gamma)$. 

Having determined $r_{M}$, we can use the inductive hypothesis to deduce that $r_1, \dots, r_{M-1}$ are also fully determined by the log as follows. 
Let $(D^{*},\calS^{*},B^{*},\Gamma^{*})$ be the $(M-1)$-log resulting from the execution of the algorithm 
for $M-1$ iterations on random input $r_1, \dots, r_{M-1}$. 
By induction, the latter sequence is uniquely determined by the $(M-1)$-log $(D^{*},\calS^{*},B^{*},\Gamma^{*})$. 
Hence, it is enough to show that $(D^{*},\calS^{*},B^{*},\Gamma^{*})$ 
is in turn uniquely determined by our initial $M$-log $(D,\calS,B,\Gamma)$. 
Clearly, 
\begin{align*}
D^* &= (d_1,\ldots, d_{M-1}), \\
B^* &= (b_1,\ldots,b_{M-1}), \\
\Gamma^* &= (\Gamma_1,\ldots, \Gamma_{M-1}).
\end{align*}
As for $\calS^{*}$, it is simply obtained from $\calS$ by letting the value of each vertex $v\in \up(u)$ 
be $\undefined$. That is, 
$$
\calS^{*}(v) = \left\{ 
\begin{array}{ll}
\undefined & \textrm{ if } v\in \up(u); \\[.5ex]
\calS(v) & \textrm{ otherwise}
\end{array}
\right.
$$
for each $v \in V(A)$ (recall that $u\in \up(u)$). 

In the case when $d_M = d_{M-1}-r+1$ with $r>0$, a near repetition was retracted during the $M$-th iteration with a repeated part of size $r$.
Here we first show 
that $r_1, \dots, r_{M-1}$ are uniquely determined, 
and then we prove that the same holds for $r_{M}$. 

Let $(D^{*},\calS^{*},B^{*},\Gamma^{*})$ denote the $(M-1)$-log  resulting from the execution of the algorithm for $M-1$ iterations on random input $r_1, \dots, r_{M-1}$. 
We thus have:
\begin{align*}
D^* &= (d_1,\ldots, d_{M-1}), \\
B^* &= (b_1,\ldots,b_{M-1}), \\
\Gamma^* &= (\Gamma_1,\ldots, \Gamma_{M-1}). 
\end{align*}  
Let us show that $\calS^{*}$ is completely determined by  $(D,\calS,B,\Gamma)$. 
Let $u$ denote the current vertex at the beginning of the $M$-th iteration. 
(Remark: The vertex $u$ can be deduced from the log $(D,\calS,B,\Gamma)$, as follows from the discussion below.)  
During the $M$-th iteration the sublist $S_{u}$ of $u$ was assigned the $r_{M}$-th $\ell$-subset of $L_{u}$. 
This triggered the existence of a bad path $v_1,\ldots,v_{2r+g}$ with $v_{2r+g}=u$, which was subsequently retracted, i.e.\ $S_{v}$ was then set to $\undefined$ for all $v\in \up(v_{r+g+1})$. 

Now, we determine the first $r+g+1$ vertices $v_1,\ldots,v_{r+g+1}$ of the bad path from the log. 
First, we show that $v_{r+g+1}$ is easily determined:   
Since $\calS^{*}$ differs from $\calS$ only on the vertices from $\up(v_{r+g+1})$, vertex $v_{r+g+1}$ is the first vertex $v$ of $A$ in depth-first left-to-right order 
from the root with $\calS(v) = \undefined$.  
(We remark that $v_{r+g+1}$ would have been the next `current vertex' considered by the algorithm if it were run for an extra iteration.) 

Next, we determine the vertices $v_1,\ldots,v_{r+g}$ from the log.
Let $w$ be the first vertex on the rightmost path of $A$ that is encountered when walking down towards the root of $A$ 
from $v_{r+g+1}$ (note that we could have $w=v_{r+g+1}$). 
By definition of the sequence $B$, among the $r+g$ vertices $v_1, \ldots, v_{r+g}$ the first $b_{M}$ vertices are on the rightmost path of $A$, and none of the remaining $r+g-b_M$ vertices are.  
Since $b_{M} \geq 1$, we may identify vertex $v_1$ by starting at vertex $w$ and walking down towards the root of $A$ either $b_M$ steps (if $w=v_{r+g+1}$) or $b_M-1$ steps (if $w \neq v_{r+g+1}$).      
Now, since we know vertices $v_{1}$, $v_{r+g+1}$ and the integer $r$, 
looking at the path from $v_{1}$ to $v_{r+g+1}$ gives us all intermediate vertices $v_{2}, \dots, v_{r+g}$ (if any), 
as well as the size $g$ of the gap section. 
Hence, $v_1, \dots, v_{r+g+1}$ are determined by the log $(D,\calS,B,\Gamma)$, as claimed.  

Building on this, we now complete the proof.    
Let $\Gamma_{M}=(\gamma_{1}, \dots, \gamma_{r})$. 
We first consider a simple case, namely $r=1$.
Then $u=v_{r+g+1}$, and at the beginning of the $M$-th iteration, 
the $r_{M}$-th $\ell$-sublist $X$ of $L_{u}$ was assigned to variable $S_{u}$, which triggered the existence 
of the bad path. 
By the definition of the sequence $\Gamma_M$, the sublist $X$ is the $\gamma_{1}$-th $\ell$-sublist of $L_{v_{r+g+1}}$ having a non-empty intersection with $\calS(v_{1})$. 
We can thus deduce $X$ from the log $(D,\calS,B,\Gamma)$, and obtain in turn $r_{M}$ from $X$. 
Notice that in this case the sublist assignments to vertices of $A$ at the end of the $(M-1)$-th and at the end of the $M$-th iterations are exactly the same, that is, $\calS^* = \calS$. 
(Indeed, this is what makes the case $r=1$ simpler.) 
Hence, $(D^{*},\calS^{*},B^{*},\Gamma^{*})$ is completely determined by $(D,\calS,B,\Gamma)$. 
By induction, the sequence $r_1, \dots, r_{M-1}$ is uniquely determined by $(D^{*},\calS^{*},B^{*},\Gamma^{*})$, 
and therefore it is also  uniquely determined by $(D,\calS,B,\Gamma)$. 
Since we have seen that $r_{M}$ is uniquely determined as well, this 
concludes the $r=1$ case. 

Now, suppose that $r \geq 2$. 
We can obtain $\calS^{*}(v_{r+g+1})$ from $(D,\calS,B,\Gamma)$, since $\calS^{*}(v_{r+g+1})$ is the $\gamma_{1}$-th $\ell$-sublist of $L_{v_{r+g+1}}$ having a non-empty intersection with $\calS(v_{1})$. 
Knowing $\calS^{*}(v_{r+g+1})$, we can identify vertex $v_{r+g+2}$ as follows. 
Consider the last time the variable $S_{v_{r+g+1}}$ was modified during the execution of the algorithm before the $M$-th iteration, say this is during the  $p_1$-th iteration. 
Thus, during that iteration, $S_{v_{r+g+1}}$ was assigned the set $\calS^{*}(v_{r+g+1})$, and 
this did not trigger the existence of a bad path. 
Then, the children $w_1, \dots, w_k$ of $v_{r+g+1}$ were inspected one by one in order, until a problematic child $w_j$ was found. 
This problematic child $w_j$ is vertex $v_{r+g+2}$. 
This process is completely deterministic, thus we can simulate it. (Indeed, we know the whole sublist assignment for vertices of $A$ at the beginning of the $p_1$-th iteration, and we know the sublist that was sampled for $v_{r+g+1}$ 
during that iteration, namely, $\calS^{*}(v_{r+g+1})$). 
Thus, for each vertex $v\in \up(w_1) \cup \cdots \cup \up(w_{j-1})$, we can figure out how the sublist $S_v$ was set during that iteration, and this is exactly the value of $\calS^{*}(v)$, since those sublists have not been modified afterwards prior to iteration $M$. 
Notice also that $\calS^{*}(v)=\undefined$ for all $v\in \up(w_{j+1}) \cup \cdots \cup \up(w_{r})$, so it only remains to determine $\calS^{*}(v)$ for $v\in \up(v_{r+g+2})$.

We can iterate this argument and discover step by step vertices $v_{r+g+2}, \dots, v_{2r+g}$ and the missing entries of $\calS^{*}$.
We spell out the general argument now, for the sake of completeness 
(the reader who is already convinced that it can be done is invited to skip this paragraph). 
For each index $i = 2, \dots, r-1$, we proceed as follows. 
The set  $\calS^{*}(v_{r+g+i})$ is the $\gamma_{i}$-th $\ell$-sublist of $L_{v_{r+g+i}}$ having a non-empty intersection with $\calS(v_{i})$, and is thus determined by  $(D,\calS,B,\Gamma)$. 
Knowing $\calS^{*}(v_{r+g+i})$, we now identify vertex $v_{r+g+i+1}$. 
Say $S_{v_{r+g+i}}$ was modified for the last time during the $p_i$-th iteration of the algorithm before the $M$-th iteration. 
During that iteration, variable $S_{v_{r+g+i}}$ was assigned the set $\calS^{*}(v_{r+g+i})$, and 
this did not result in the existence of any bad path.  
Next, the children $w_1, \dots, w_k$ of $v_{r+g+i}$ were inspected until a problematic child $w_j$ was found, 
which is vertex $v_{r+g+i+1}$. 
We know the whole sublist assignment for vertices of $A$ at the beginning of the $p_i$-th iteration; indeed, 
this is exactly the one at the {\em end} of the $p_{i-1}$-th iteration (as is easily checked), which we already know.  
We also know the sublist that was sampled for vertex $v_{r+g+i}$ during the $p_i$-th iteration. 
Hence, we can simulate the execution of lines~\ref{algo:choice-extension-begin}--\ref{algo:choice-extension-end} 
of the algorithm for the $p_i$-th iteration. 
This implies that, for each vertex $v\in \up(w_1) \cup \cdots \cup \up(w_{j-1})$, we can determine how the sublist $S_v$ was set during that iteration, and that sublist is precisely the set $\calS^{*}(v)$.
Also, we have that $\calS^{*}(v)=\undefined$ for all $v\in \up(w_{j+1}) \cup \cdots \cup \up(w_{r})$. 

This way we completely determined $\calS^{*}(v)$ for all vertices $v \in \up(v_{r+g+i}) - \up(v_{r+g+i +1})$ for each $i \in [r-1]$. 
For all other vertices $v$ of $A$, we have $\calS^{*}(v)=\calS(v)$.
Thus, $\calS^{*}$ is completely determined by $(D,\calS,B,\Gamma)$, and hence so is the $(M-1)$-th log  $(D^{*},\calS^{*},B^{*},\Gamma^{*})$, as claimed.  

Equipped with the knowledge of $\calS^{*}$, we may now finish the proof in a manner similar to the $r=1$ case. 
Let $X$ denote the $r_{M}$-th $\ell$-sublist $X$ of $L_{v_{2r+g}}$. 
The sublist $X$  is the $\gamma_{r}$-th $\ell$-sublist of $L_{v_{2r+g}}$ having a non-empty intersection with $\calS(v_{r})$, which is thus determined by the log $(D,\calS,B,\Gamma)$. 
Hence, we can obtain $r_M$ from $(D,\calS,B,\Gamma)$. 
Moreover, as we have seen, $(D^{*},\calS^{*},B^{*},\Gamma^{*})$ is completely determined by the log $(D,\calS,B,\Gamma)$. 
By induction, the sequence $r_1, \dots, r_{M-1}$ is uniquely determined by $(D^{*},\calS^{*},B^{*},\Gamma^{*})$, 
and therefore it is uniquely determined by $(D,\calS,B,\Gamma)$. 
This concludes the proof. 
\end{proof}

We remark that no effort has been made to optimize the bound of $32\ell^{3} + 1$ in Lemma~\ref{lemma:more-technical}. 

\section*{Acknowledgements} 
We thank the two anonymous referees for their insightful comments, which greatly helped us improve the exposition.    



\end{document}